\documentclass[preprint,12pt]{elsarticle}
\usepackage{calc}
\usepackage{tikz}
\usetikzlibrary{snakes,arrows}
\usetikzlibrary{graphs,graphs.standard}
\usetikzlibrary{decorations.markings}
\tikzstyle{vertex}=[circle, draw, inner sep=0pt, minimum size=6pt]
\newcommand{\vertex}{\node[vertex]}

\usepackage{amsmath,amsthm,mathrsfs,array,subfigure,graphicx}
\usepackage{amsfonts}
\usepackage{amsthm}
\usepackage{amssymb}
\usepackage{graphicx}
\usepackage{tabularx}
\usepackage{caption} 
\usepackage{amsmath}
\usepackage[figure,tworuled,linesnumbered,vlined]{algorithm2e}

\usepackage{tikz}
\usepackage{tikz,pgfplots}
\usetikzlibrary{decorations.markings}
\usepackage{color}

\newtheorem{theorem}{Theorem}[section]

\newtheorem{lemma}[theorem]{Lemma}
\newtheorem{proposition}[theorem]{Proposition}

\newtheorem{problem}[theorem]{Problem}

\newtheorem{corollary}[theorem]{Corollary}
\newtheorem{definition}[theorem]{Definition}

\newcommand{\s}{{\rm s}}
\newcommand{\e}{{\rm e}}

\newcommand{\vp}{{\rm vp}}
\newcommand{\gp}{{\rm gp}}
\newcommand{\rad}{{\rm rad}}
\newcommand{\ext}{{\rm Ext}}
\newcommand{\diam}{{\rm diam}}

\footskip=30pt \vspace{5cm}

\numberwithin{figure}{section}

\begin{document}

	\begin{frontmatter}
		
		
		\title{On the Vertex Position Number of Graphs}
		

		\author[label1]{Maya G.S. Thankachy}
		\ead{mayagsthankachy@gmail.com}
		\author[label2]{Elias John Thomas}
		\ead{eliasjohnkalarickal@gmail.com}
		\author[label1]{Ullas Chandran S. V.}
		\ead{svuc.math@gmail.com}
		\author[label3]{James Tuite}
		\ead{james.t.tuite@open.ac.uk}
		\author[label4]{Gabriele {Di Stefano}}
		\ead{gabriele.distefano@univaq.it}
		\author[label3]{Grahame Erskine}
		\ead{grahame.erskine@open.ac.uk}
		
	\address[label1]{Department of Mathematics, Mahatma Gandhi College, University of Kerala, Thiruvananthapuram-695004, Kerala, India}
		
		\address[label2]{Department of Mathematics, Mar Ivanios College, University of Kerala, Thiruvananthapuram-695015, Kerala, India}
		
		\address[label3]{Department of Mathematics and Statistics, Open University, Walton Hall, Milton Keynes, UK}

		\address[label4]{Department of Information Engineering, Computer Science and Mathematics, University of L'Aquila,  Italy}

		\begin{abstract}
		In this paper we generalise the notion of visibility from a point in an integer lattice to the setting of graph theory. For a vertex $x$ of a connected graph $G$, we say that a set $S \subseteq V(G)$ is an \emph{$x$-position set} if for any $y \in S$ the shortest $x,y$-paths in $G$ contain no point of $S\setminus \{ y\}$. We investigate the largest and smallest orders of maximum $x$-position sets in graphs, determining these numbers for common classes of graphs and giving bounds in terms of the girth, vertex degrees, diameter and radius. Finally we discuss the complexity of finding maximum vertex position sets in graphs. 
		\end{abstract}
		
		\begin{keyword}
			geodesic \sep vertex position set \sep vertex position number \sep general position
			
			\MSC 05C12 \sep 05C69 
		\end{keyword}
		
	\end{frontmatter}


\section{Introduction}\label{sec:Intro}
All graphs considered in this paper are finite, undirected and simple.  For a graph $G$ we will denote the subgraph induced by a subset $S \subseteq V(G)$ by $G[S]$. The \emph{distance} $d(u,v)$ between two vertices $u$ and $v$ in a connected graph $G$ is the length of a shortest $u,v$-path in $G$; any such path is called a \emph{geodesic}. The distance is a metric on the vertex set $V$. The diameter $\diam(G)$ of a connected graph $G$ is the length of any longest geodesic. For any vertex $u$ of $G$, the \emph{eccentricity} of $u$ is $\e(u) = \max \{d(u, v) : v \in V \} $. A vertex $v$ of $G$ such that $d(u,v) = \e(u)$ is called an \emph{eccentric vertex} of $u$. The neighborhood of a vertex $v$ is the set $N(v)$ consisting of all vertices $u$ which are adjacent with $v$. A vertex $v$ is \emph{simplicial} if the subgraph induced by its neighborhood $N(v)$ is complete; we will denote the number of simplicial vertices of a graph $G$ by $\s(G)$ and the set of all simplicial vertices of $G$ by $\ext (G)$. A set of vertices in a graph is \emph{independent} if no two vertices in the set are adjacent; the independence number $\alpha (G)$ is the number of vertices in a largest independent set of $G$. A graph is a \emph{block graph} if every maximal $2$-connected component is a clique. For basic graph theoretic terminology not defined here we refer to~\cite{rgbuck, rcbook}.

Visibility and illumination problems are among the most attractive and interesting research topics in combinatorics, geometry and number theory~\cite{boltjansky}. Such problems have been studied intensively in the context of the integer lattice; a set $\Lambda $ of points of the lattice is \emph{visible} from a point $x$ if for any $y \in \Lambda $ the line segment from $x$ to $y$ contains no other points of $\Lambda $. A well-known result from elementary analytic number theory first proved by Sylvester~\cite{Syl} states that the density of the set of points in the integer lattice that are visible from the origin is $\frac{6}{\pi ^2}$~\cite{Apo}. In particular, in Chapter III of~\cite{Hardy-2008}, it is shown how to place a set of points with positive integer coordinates $(i,j)$, $j\leq i$, in such a way that each point is visible from the origin $(0,0)$, by also maximising the number of points with the same abscissa. This construction has interesting relations with the Farey series and Euler's totient function $\phi$. Other interesting visibility problems in integer lattices can be found in~\cite{erdos, hammer}.

In recent years the algorithmic component of visibility problems has attracted great attention under the name \emph{art gallery} or \emph{watchman} problems, which lie in the intersection of combinatorial and computational geometry~\cite{O'Rouke}. Art gallery problems, theorems and algorithms are so named after the celebrated 48 years old problem posed by V. Klee. In 1973 he asked the following question: `What is the minimum number of guards sufficient to cover the interior of an $n$-wall gallery?'. This problem was solved by Chva'tal and subsequently by Fisk. By creating idealised situations such as obstacles, guards, etc., the theory succeeds in abstracting the algorithmic essence of many visibility problems.
  
Taking our inspiration from the result of Sylvester~\cite{Syl}, in this paper we consider a generalisation of `local visibility problems' to the context of the general position problem in graph theory. The general position problem originated in Dudeney's no-three-in-line problem~\cite{dudeney-1917} and the general position subset selection problem~\cite{froese-2017, payne-2013} from discrete geometry. These problems were generalised to graphs independently in~\cite{ullas-2016} and~\cite{manuel-2018a}. A set $S$ of vertices of a graph $G$ is in \emph{general position} if for any $u,v \in S$ any $u,v$-geodesic does not intersect $S \setminus \{ u,v\} $. The \emph{general position number} $\gp (G)$ of $G$ is the number of vertices in a largest general position set in $G$. We refer the reader to~\cite{anand-2019,manuel-2018b,elias-2020} for more information on the general position problem.   

In a recent paper Di Stefano~\cite{DiStefano-2022} introduced the concept of a \emph{mutual visibility set} in a graph; a set $S$ of points in a graph $G$ are \emph{mutually visible} if for any $u,v \in S$ there exists a shortest $u,v$-path in $G$ that does not pass through $S \setminus \{ u,v\} $; the \emph{mutual visibility number} $\mu (G)$ of $G$ is the number of vertices in a largest mutual visibility set in $G$. In~\cite{DiStefano-2022} the mutual visibility number of some classes of graphs are determined and it is shown that the problem of finding a largest mutual visibility set is NP-complete for general graphs. 

We now study a `local' version of these problems using a parameter that we call the \emph{vertex position number} of a graph. The plan of this paper is as follows. In Section~\ref{Section: vertex position sets} we provide some bounds on the vertex position numbers of a graph. In Section~\ref{Section: graph classes} the vertex position numbers of some common classes of graphs are determined. We characterise the graphs with very large or small vertex position numbers in Section~\ref{Section: characterisation}. Finally in Section~\ref{Section: complexity} we consider the computational complexity of finding the vertex position number of a graph.

\section{Vertex position sets in graphs}\label{Section: vertex position sets}
In this section we derive bounds for the vertex position numbers of a graph in terms of the minimum and maximum degrees, radius and diameter. First we formally define the vertex position numbers.

\begin{definition}\label{Def: main definition}
For any graph $G$ and a fixed vertex $x \in V(G)$, a set $S_x\subseteq V(G)$ is an \emph{$x$-position set} if for any $y\in S_x$ no vertex of $S_x\setminus \{ y\} $ lies on any $x,y$-geodesic in $G$. The \emph{$x$-position number} of $G$ is defined to be the maximum cardinality of an \emph{$x$-position set} and is denoted by \emph {$p_x(G)$} or simply \emph{$p_x$}. An \emph{$x$-position set} of cardinality \emph{$p_x(G)$} is called a \emph{$p_x$-set}. The maximum value of $p_x(G)$ among all vertices $x$ of $G$ is called the \emph{upper vertex position number} $\vp(G)$ (or simply the \emph{vertex position number}) of $G$. Similarly, the minimum value of $p_x(G)$ among all vertices of $G$ is called the \emph{minimum vertex position number} $\vp^-(G)$ of $G$.
\end{definition}
To illustrate these concepts, consider the graph $G$ in Figure~\ref{fig2.1}. We give the $x$-position numbers of $G$ for representative vertices in Table~\ref{Table 1}, together with a (not necessarily unique) $x$-position set. We see from the table that $\vp^-(G) = 4$ and $\vp(G) = 11$.

\begin{figure}
			\centering
			\begin{tikzpicture}[x=0.4mm,y=-0.4mm,inner sep=0.2mm,scale=0.7,very thick,vertex/.style={circle,draw,minimum size=15,fill=white}]

				\node at (-150,0) [vertex] (c1) {$c_1$};
				\node at (-50,0) [vertex] (c2) {$c_2$};
				\node at (50,0) [vertex] (c3) {$c_3$};
				\node at (150,0) [vertex] (c4) {$c_4$};
				
				\node at (-150,50) [vertex] (b1) {$b_1$};
				\node at (-50,50) [vertex] (b2) {$b_2$};
				\node at (50,50) [vertex] (b3) {$b_3$};
				\node at (150,50) [vertex] (b4) {$b_4$};
				
				\node at (-150,100) [vertex] (a1) {$a_1$};
				\node at (-50,100) [vertex] (a2) {$a_2$};
				\node at (50,100) [vertex] (a3) {$a_3$};
				\node at (150,100) [vertex] (a4) {$a_4$};

				\node at (0,-50) [vertex] (x) {$x$};
				
				\path
				
				(a1) edge (b2)
				(a1) edge (b3)
				(a1) edge (b4)
				(a2) edge (b1)
				(a2) edge (b3)
				(a2) edge (b4)
				(a3) edge (b1)
				(a3) edge (b2)
				(a3) edge (b4)
				(a4) edge (b1)
				(a4) edge (b2)
				(a4) edge (b3)
				
				(c1) edge (b2)
				(c1) edge (b3)
				(c1) edge (b4)
				(c2) edge (b1)
				(c2) edge (b3)
				(c2) edge (b4)
				(c3) edge (b1)
				(c3) edge (b2)
				(c3) edge (b4)
				(c4) edge (b1)
				(c4) edge (b2)
				(c4) edge (b3)
				
				(c1) edge (x)
				(c2) edge (x)
				(c3) edge (x)
				(c4) edge (x)
				
				(c1) edge (b1)
				(c2) edge (b2)
				(c3) edge (b3)
				(c4) edge (b4)
				(a1) edge (b1)
				(a2) edge (b2)
				(a3) edge (b3)
				(a4) edge (b4)
				
				(c1) edge (c2)
				(c2) edge (c3)
				(c3) edge (c4)
				(c1) edge[bend left = 20] (c3)
				(c1) edge[bend left = 20] (c4)
				(c2) edge[bend left = 20] (c4)
				
				(b1) edge (b2)
				(b2) edge (b3)
				(b3) edge (b4)
				(b1) edge[bend left = 20] (b3)
				(b1) edge[bend left = 20] (b4)
				(b2) edge[bend left = 20] (b4)
				
				(a1) edge (a2)
				(a2) edge (a3)
				(a3) edge (a4)
				(a1) edge[bend left = 20] (a3)
				(a1) edge[bend left = 20] (a4)
				(a2) edge[bend left = 20] (a4)
				;
			\end{tikzpicture}
			\caption{}
			\label{fig2.1}
		\end{figure}

\begin{table}[ht]
\centering
\begin{tabular}{|c| c| c|}
\hline\hline
Vertex & $p_x$-set & $p_x(G)$ \\
\hline
$x$ & $\{ c_1,c_2,c_3,c_4\} $ & $4$\\
\hline
$c_1$ & $\{ x,c_2,c_3,c_4,b_1,b_2,b_3,b_4 \}$ & $8$\\
\hline
$b_1$ & $\{ c_1,c_2,c_3,c_4,b_2,b_3,b_4,a_1,a_2,a_3,a_4 \} $ & $11$\\
\hline
$a_1$ & $\{ a_2,a_3,a_4,b_1,b_2,b_3,b_4\}$ & $7$\\
\hline
\end{tabular}
\caption{ }
\label{Table 1}
\end{table}

Unless stated otherwise (for example in Theorem~\ref{thm:Nordhasu Gaddum}) we assume all graphs to be connected. However, Definition~\ref{Def: main definition} also applies to disconnected graphs; if $x$ belongs to a component $C$ of a disconnected graph $G$, then any vertex $y$ from another component $D$ of $G$ can be included in an $x$-position set, as there is no $x,y$-path in $G$. Hence in this case $p_x(G) = (n-|C|)+p_x(C)$. 

For any vertex $x \in V(G)$ the set $\{ x\}$ is an $x$-position set; however, by the convention in Definition~\ref{Def: main definition}, if $G$ is connected, then $x$ is not contained in any $x$-position set of order $\geq 2$. Hence for any connected graph $G$ with order $n \geq 2$ we have $1 \leq p_x(G) \leq n-1$ for any $x \in V(G)$ and, more generally, a (not necessarily connected) graph $G$ has $\vp (G) = n$ if and only if $G$ has an isolated vertex. These bounds are sharp: for any path $P_n$ of length $\geq 1$ we have $p_x(P_n) = 1$ for either terminal vertex, whilst for $n \geq 2$ we have $p_x(K_n)=n-1$ for every vertex of a complete graph $K_n$.  In this section we derive several bounds for the vertex position numbers in terms of various graph parameters. First we compare the vertex position number with the general position number.
 
\begin{lemma}\label{lower bound vp vs gp}
The vertex position number and general position number of a graph are related by $\vp(G) \geq \gp(G)-1$.
\end{lemma}
\begin{proof}
Let $S$ be a gp-set of $G$ with $|S| = \gp (G)$. Choose a vertex $x \in S$. Then $S\setminus \{ x\} $ is an $x$-position set, implying that $\vp (G) \geq p_x(G) \geq \gp(G) -1$. 
\end{proof}
The bound in Lemma~\ref{lower bound vp vs gp} is met by the complete graph $K_n$. However, we now give an example to show that the numbers $\vp ^-(G)$, $\gp (G)$ and $\vp (G)$ can be arbitrarily far apart. For $r \geq 2$ we define the vertex set of the graph $G(r)$ to be $\{ u_{i,j}: 1 \leq i \leq 7, 1 \leq j \leq r \} \cup \{ x\} $. Let $H(r)$ be the graph on the same vertex set as $G(r)$ with adjacencies defined as follows: 
\begin{itemize}
    \item $x \sim u_{1,j}$ for $1 \leq j \leq r$,
    \item $u_{i,j} \sim u_{i,j'}$, $1 \leq i \leq 7$, $1 \leq j,j' \leq r$ and $j \not = j'$, and
    \item $u_{i,j} \sim u_{i+1,j'}$ for $1 \leq i \leq 6$ and $1 \leq j,j' \leq r$.
\end{itemize}
Now define $G(r)$ to be the graph formed by deleting all edges to the vertices $u_{3,1}$, $u_{4,1}$ and $u_{5,1}$ except for the path $u_{2,1},u_{3,1},u_{4,1},u_{5,1},u_{6,1}$. See Figure~\ref{fig:G(r)} for an example. The minimum vertex position number of this graph is $r$ (attained at the vertex $x$) and the vertex position number is $6r-4$ (attained at the vertex $u_{4,1}$).

\begin{figure}
			\centering
			\begin{tikzpicture}[x=0.4mm,y=-0.4mm,inner sep=0.2mm,scale=1.0,very thick,vertex/.style={circle,draw,minimum size=7,fill=white}]
				
				\node at (-200,0) [vertex,color=green] (x) {};
				
				\node at (-150,75) [vertex,color=red] (u11) {};
				\node at (-150,25) [vertex,color=red] (u12) {};
				\node at (-150,-25) [vertex,color=red] (u13) {};
				\node at (-150,-75) [vertex,color=red] (u14) {};
				
				\node at (-100,75) [vertex] (u21) {};
				\node at (-100,25) [vertex,color=red] (u22) {};
				\node at (-100,-25) [vertex,color=red] (u23) {};
				\node at (-100,-75) [vertex,color=red] (u24) {};
				
				\node at (-50,75) [vertex] (u31) {};
				\node at (-50,25) [vertex,color=red] (u32) {};
				\node at (-50,-25) [vertex,color=red] (u33) {};
				\node at (-50,-75) [vertex,color=red] (u34) {};
				
				\node at (0,75) [vertex,color=blue] (u41) {};
				\node at (0,25) [vertex] (u42) {};
				\node at (0,-25) [vertex] (u43) {};
				\node at (0,-75) [vertex] (u44) {};
				
				\node at (50,75) [vertex] (u51) {};
				\node at (50,25) [vertex,color=red] (u52) {};
				\node at (50,-25) [vertex,color=red] (u53) {};
				\node at (50,-75) [vertex,color=red] (u54) {};
			
			    \node at (100,75) [vertex] (u61) {};
				\node at (100,25) [vertex,color=red] (u62) {};
				\node at (100,-25) [vertex,color=red] (u63) {};
				\node at (100,-75) [vertex,color=red] (u64) {};
				
				\node at (150,75) [vertex,color=red] (u71) {};
				\node at (150,25) [vertex,color=red] (u72) {};
				\node at (150,-25) [vertex,color=red] (u73) {};
				\node at (150,-75) [vertex,color=red] (u74) {};

				\path
				
				(x) edge (u11)
				(x) edge (u12)
				(x) edge (u13)
				(x) edge (u14)
				
				(u11) edge (u12)
				(u11) edge[bend left] (u13)
				(u11) edge[bend left] (u14)
				(u12) edge (u13)
				(u12) edge[bend left] (u14)
				(u13) edge (u14)
				
				(u21) edge (u22)
				(u21) edge[bend left] (u23)
				(u21) edge[bend left] (u24)
				(u22) edge (u23)
				(u22) edge[bend left] (u24)
				(u23) edge (u24)

				(u32) edge (u33)
				(u32) edge[bend left] (u34)
				(u33) edge (u34)

				(u42) edge (u43)
				(u42) edge[bend left] (u44)
				(u43) edge (u44)

				(u52) edge (u53)
				(u52) edge[bend left] (u54)
				(u53) edge (u54)
			
				(u61) edge (u62)
				(u61) edge[bend left] (u63)
				(u61) edge[bend left] (u64)
				(u62) edge (u63)
				(u62) edge[bend left] (u64)
				(u63) edge (u64)
				
				(u71) edge (u72)
				(u71) edge[bend left] (u73)
				(u71) edge[bend left] (u74)
				(u72) edge (u73)
				(u72) edge[bend left] (u74)
				(u73) edge (u74)
				
				(u11) edge (u21)
				(u11) edge (u22)
				(u11) edge (u23)
				(u11) edge (u24)
				
				(u12) edge (u21)
				(u12) edge (u22)
				(u12) edge (u23)
				(u12) edge (u24)
				
				(u13) edge (u21)
				(u13) edge (u22)
				(u13) edge (u23)
				(u13) edge (u24)
				
				(u14) edge (u21)
				(u14) edge (u22)
				(u14) edge (u23)
				(u14) edge (u24)
				
				(u21) edge (u31)
				(u21) edge (u32)
				(u21) edge (u33)
				(u21) edge (u34)

				(u22) edge (u32)
				(u22) edge (u33)
				(u22) edge (u34)

				(u23) edge (u32)
				(u23) edge (u33)
				(u23) edge (u34)

				(u24) edge (u32)
				(u24) edge (u33)
				(u24) edge (u34)

				(u31) edge (u41)
				
				(u32) edge (u42)
				(u32) edge (u43)
				(u32) edge (u44)

				(u33) edge (u42)
				(u33) edge (u43)
				(u33) edge (u44)

				(u34) edge (u42)
				(u34) edge (u43)
				(u34) edge (u44)
				
				(u41) edge (u51)
				
				(u42) edge (u52)
				(u42) edge (u53)
				(u42) edge (u54)

				(u43) edge (u52)
				(u43) edge (u53)
				(u43) edge (u54)

				(u44) edge (u52)
				(u44) edge (u53)
				(u44) edge (u54)

				(u51) edge (u61)

				(u52) edge (u61)
				(u52) edge (u62)
				(u52) edge (u63)
				(u52) edge (u64)
				
				(u53) edge (u61)
				(u53) edge (u62)
				(u53) edge (u63)
				(u53) edge (u64)
				
				(u54) edge (u61)
				(u54) edge (u62)
				(u54) edge (u63)
				(u54) edge (u64)

				(u61) edge (u71)
				(u61) edge (u72)
				(u61) edge (u73)
				(u61) edge (u74)
				
				(u62) edge (u71)
				(u62) edge (u72)
				(u62) edge (u73)
				(u62) edge (u74)
				
				(u63) edge (u71)
				(u63) edge (u72)
				(u63) edge (u73)
				(u63) edge (u74)
				
				(u64) edge (u71)
				(u64) edge (u72)
				(u64) edge (u73)
				(u64) edge (u74)
				;
			\end{tikzpicture}
			\caption{A graph with $\vp ^- = r$ (green vertex) and $\vp = 6r-4$ (blue vertex). A largest $\vp $-set of the blue vertex is shown in red.}
			\label{fig:G(r)}
		\end{figure}
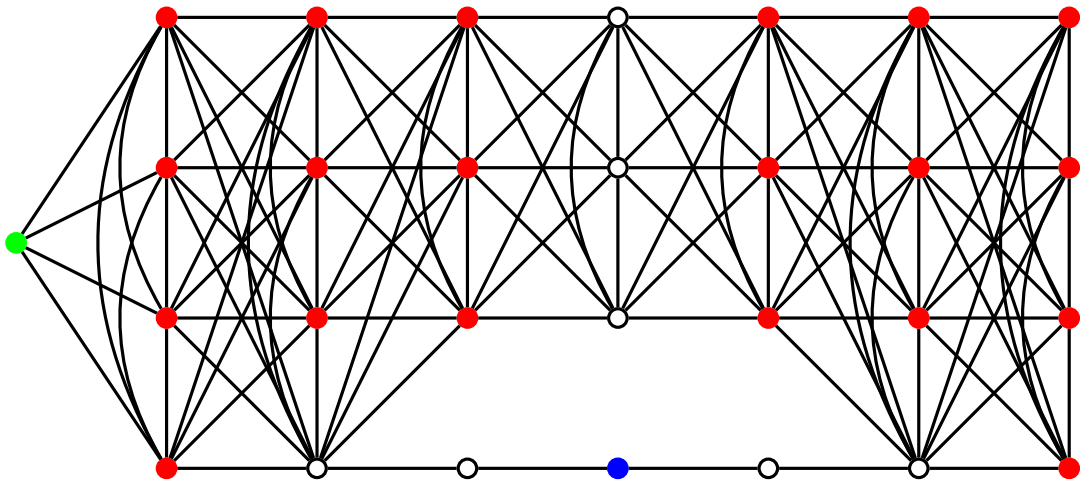

\begin{lemma}
For $r \geq 3$, we have $\vp^-(G(r)) = r, \gp(G(r)) = 2r, \vp(G(r)) = 6r-4$.
\end{lemma}

This raises the question of how far apart the numbers $\vp ^-(G)$ and $\vp (G)$ can be.

\begin{problem}
Is the ratio $\frac{\vp (G)}{\vp^- (G)} $ bounded for connected graphs?
\end{problem}

We now bound the vertex position numbers in terms of the vertex degrees.

\begin{lemma}\label{lem:degree bound}
Let $G$ have minimum degree $\delta $ and maximum degree $\Delta $. Then $\vp^-(G) \geq \delta $ and $\vp (G) \geq \Delta $.
\end{lemma}
\begin{proof}
It follows from Definition~\ref{Def: main definition} that for any vertex $x$ of $G$ the neighbourhood $N(x)$ is an $x$-position set of $G$. Therefore for all vertices $x \in V(G)$ we have $p_x(G) \geq d(x)$. This implies that all $p_x$-sets have order at least $\delta $ and there exists a $p_x$-set with order at least $\Delta $.
\end{proof}

We now generalise this result to sets of vertices at given distance from a fixed vertex; this leads to bounds on $p_x(G)$ in terms of the order $n$ of $G$ and the eccentricity $\e(x)$ of the vertex $x$. 

\begin{lemma}\label{eccentricity bound}
	For any vertex $x$ of $G$ with eccentricity $\e(x)$ the vertex position number satisfies $p_x(G) \geq \frac{n-1}{\e(x)}$. Thus $\vp^-(G) \geq \frac{n-1}{\diam(G)}$ and $\vp(G) \geq \frac{n-1}{\rad(G)}$.
\end{lemma}
\begin{proof}
	For $1 \leq t \leq \e(x)$, let $V_t$ be the set of vertices at distance exactly $t$ from $x$ in $G$. Each of the sets $V_t$ is an $x$-position set and one of them must have order at least $\frac{n-1}{\e(x)}$, so that $p_x(G) \geq \frac{n-1}{\e(x)}$. Since $\rad(G) \leq \e(x)\leq \diam(G)$ for all $x \in V(G)$ the result follows. 
\end{proof}

This bound immediately characterises the graphs with vertex position number one.

\begin{corollary}\label{cor:path}
A graph $G$ satisfies $\vp^-(G) = 1$ if and only if $G$ is a path. The only connected graphs with $\vp (G) = 1$ are $K_1$ and $K_2$.
\end{corollary}
\begin{proof}
If $x$ is the terminal vertex of a path $P_n$, then $p_x(P_n) = 1$. Conversely, if $\vp^-(G) = 1$, then by Lemma~\ref{eccentricity bound} we must have $\diam (G) = n-1$, which implies that $G$ is a path. By Lemma~\ref{lem:degree bound} any graph $G$ with $\vp (G) = 1$ must have maximum degree $\Delta \leq 1$, which proves the latter statement. \end{proof}

The argument of Lemma~\ref{eccentricity bound} also easily yields the vertex position number of the join of two graphs.
\begin{corollary}
If $G_i$ has order $n_i$ and maximum degree $\Delta _i$ for $i = 1,2$, then the vertex position number of the join is \[ \vp (G\vee H)= \max \{ n_1+\Delta _2,n_2+\Delta _1\} .\]
\end{corollary}
\begin{proof}
The diameter of $G$ is two. For any vertex $x$ in $G_1$ the set $V(G_2) \cup N_{G_1}(x)$ is an $x$-position set by Lemma~\ref{eccentricity bound} with order $n_2+\Delta _1$, with a similar result for vertices $y$ in $G_2$. Suppose without loss of generality that the largest vertex position set is attained at a vertex $x \in V(G_1)$. Then we can assume that a $p_x$-set contains a vertex of $G_2$, for otherwise $\vp(G) \vee H) = p_x(G \vee H) \leq n_1 < p_y(G\vee H)$ for any $y \in V(G_2)$, a contradiction. As a maximum $p_x$-set $S_x$ contains a vertex of $G_2$, $S_x$ cannot contain any vertex of $V(G_1)\setminus (\{ x\} \cup N_{G_1}(x)$, so that $p_x(G\vee H) \leq n_2+\Delta _1$. The result follows. 
\end{proof}

\begin{theorem}
For any connected graph $G$ we have $\vp^-(G) \geq \left \lceil \frac{\Delta +1 }{3} \right \rceil$. If $G$ is bipartite, then $\vp^-(G) \geq \left \lceil \frac{\Delta }{2} \right \rceil $.
\end{theorem}
\begin{proof}
Let $G$ be a connected graph with maximum degree $\Delta $ and let $x$ be a vertex of $G$ with this degree. Let $y$ be any vertex of $G$. If $y = x$, we have $p_y(G) \geq \Delta $, so suppose that $y \not = x$ and let $r \geq 0$ be the length of the shortest path from $y$ to $N(x)$. Then it can be seen that the distance from $y$ to any vertex of $\{ x\} \cup N(x)$ is one of $r$, $r+1$ or $r+2$. It follows that one of the level sets of $G$ in the distance partition with respect to $y$ must have order at least $\frac{\Delta +1}{3}$. Hence $p_y(G) \geq \frac{\Delta +1}{3}$ and our proof is complete. If $G$ is bipartite, then $N(x)$ is an independent set and the distance from $y$ to any vertex of $N(x)$ is either $r$ or $r+2$, so we can improve the bound to $\vp^-(G) \geq \frac{\Delta }{2}$ in this case. The constructions in Figures~\ref{fig:lower bound vp^-} show that both of these bounds are tight (in both cases $x$ is a vertex with maximum degree and $p_y(G) = \vp ^-(G)$).
\end{proof}

\begin{figure}
			\centering
			\begin{tikzpicture}[x=0.4mm,y=-0.4mm,inner sep=0.2mm,scale=0.5,very thick,vertex/.style={circle,draw,minimum size=15,fill=white}]

				\node at (-150,50) [vertex,fill=red] (b1) {};
				\node at (-50,50) [vertex,fill=red] (b2) {};
				\node at (50,50) [vertex,fill=red] (b3) {};
				\node at (150,50) [vertex,fill=red] (b4) {};
				
				\node at (-150,100) [vertex] (a1) {};
				\node at (-50,100) [vertex] (a2) {};
				\node at (50,100) [vertex] (a3) {};
				\node at (150,100) [vertex] (a4) {};
				
				\node at (-150,-50) [vertex] (c1) {};
				\node at (-50,-50) [vertex] (c2) {};
				\node at (50,-50) [vertex] (c3) {};
				\node at (150,-50) [vertex] (c4) {};
				
				\node at (0,150) [vertex] (y) {$y$};
				\node at (0,0) [vertex,fill=red] (x) {$x$};
				
					\node at (200,50) [vertex,fill=red] (f1) {};
				\node at (300,50) [vertex,fill=red] (f2) {};
				\node at (400,50) [vertex,fill=red] (f3) {};
				\node at (500,50) [vertex,fill=red] (f4) {};
				
				\node at (200,-50) [vertex] (d1) {};
				\node at (300,-50) [vertex] (d2) {};
				\node at (400,-50) [vertex] (d3) {};
				\node at (500,-50) [vertex] (d4) {};
				
				\node at (350,100) [vertex] (g) {$y$};
				\node at (350,0) [vertex] (h) {$x$};
				
				\path
				
				(x) edge (b1)
				(x) edge (b2)
				(x) edge (b3)
				(x) edge (b4)
				(x) edge (a1)
				(x) edge (a2)
				(x) edge (a3)
				(x) edge (a4)
				(x) edge (c1)
				(x) edge (c2)
				(x) edge (c3)
				(x) edge (c4)
				
				(a1) edge (b1)
				(b1) edge (c1)
				(a2) edge (b2)
				(b2) edge (c2)
				(a3) edge (b3)
				(b3) edge (c3)
				(a4) edge (b4)
				(b4) edge (c4)

				(y) edge (a1)
				(y) edge (a2)
				(y) edge (a3)
				(y) edge (a4)
				
				(a1) edge (b1)
				(a2) edge (b2)
				(a3) edge (b3)
				(a4) edge (b4)
				
				(h) edge (d1)
				(h) edge (d2)
				(h) edge (d3)
				(h) edge (d4)
				(h) edge (f1)
				(h) edge (f2)
				(h) edge (f3)
				(h) edge (f4)
				
				(g) edge (f1)
				(g) edge (f2)
				(g) edge (f3)
				(g) edge (f4)

				;
			\end{tikzpicture}
			\caption{A graph with $\vp^-(G) = \left\lceil \frac{ \Delta+1 }{3}\right\rceil $ (left) and a bipartite graph with $\vp^-(G) = \frac{\Delta }{2}$ (right), with $p_y$-sets in red.}
			\label{fig:lower bound vp^-}
		\end{figure}
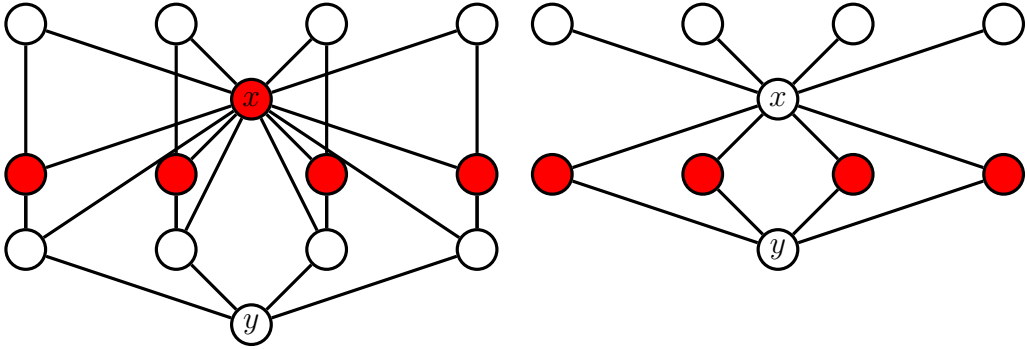
		
Now we give an upper bound for the vertex position number in terms of vertex eccentricity.
\begin{lemma}\label{lem:eccentricity upper bound}
For any vertex $x$ of $G$ with eccentricity $\e(x)$, the $x$-vertex position number of $G$ is bounded above by $p_x(G) \leq n-\e(x)$. Thus $\vp (G) \leq n-\rad (G)$ and $\vp ^-(G) \leq n-\diam (G)$.  
\end{lemma}
\begin{proof}
Fix a vertex $x$ of $G$ with eccentricity $\e(x)$. Let $y$ be an eccentric vertex of $x$, i.e. $d(x,y) = \e(x)$. Let $x = u_0,u_1,\dots,u_{\e(x)}=y$ be an $x,y$-geodesic in $G$ and $S_x$ be an $x$-position set of $G$ with order $p_x(G)$. Suppose that $u_i,u_j\in S_x$ for some $i,j$ with $0 \leq i< j \leq \e(x)$. But then $u_i$ lies on an $x,u_j$-geodesic, contradicting the fact that $S_x$ is an $x$-position set. Hence any $x$-position set contains at most one vertex from the set $\{ x,u_1,u_2,\dots,u_{\e(x)}\}$. Thus $p_x(G) = |S_x|\leq n-1-\e(x)+1 = n-\e(x)$. Thus $\vp(G) \leq n-\rad(G)$ and $\vp ^-(G) \leq n-\diam (G)$.
\end{proof}

The following theorem improves the upper bound for $\vp(G)$ in Lemma~\ref{lem:eccentricity upper bound}.
\begin{theorem}
For any graph $G$ with $\rad(G) \geq 3$ we have $\vp(G) \leq n-\rad(G)-1$.
\end{theorem}
\begin{proof}
Suppose that $G$ has radius $\rad(G) \geq 3$ and meets the upper bound in Lemma~\ref{lem:eccentricity upper bound}. Then the largest value of the vertex position number is achieved by a central vertex, call it $u$. Let $S_u$ be any $u$-position set of order $n-\rad(G)$. By the argument of Lemma~\ref{lem:eccentricity upper bound} there is a path $P$ from $u$ to one of its eccentric vertices such that $P$ contains just one vertex of $S_u$ and all vertices of $V(G)\setminus V(P)$ belong to $S_u$. As the shortest path from $u$ to any vertex $x \in V(G)\setminus V(P)$ cannot pass through another vertex of $V(G)\setminus V(P)$, the shortest path from $u$ to $x$ consists of a section of $P$ followed by an edge from $P$ to $x$.  Hence each vertex of $V(G) \setminus V(P)$ has an edge to $P$; however, this contradicts our supposition that $u$ is a central vertex.
\end{proof}

A vertex $v$ in a connected graph $G$ is a \emph{boundary vertex} of a vertex $u$ if $d(u,w)\leq d(u,v)$ for each neighbour $w$ of $v$. The set of all boundary vertices of $u$ is denoted by $\partial (u)$.

\begin{proposition}\label{pro:boundary}
For any connected graph $G$ and any vertex $x \in V(G)$, the boundary $\partial (x)$ is an $x$-position set of $G$.
\end{proposition}
\begin{proof}
Assume to the contrary that $\partial (x)$ is not an $x$-position set; hence there must be a geodesic $x = u_0,u_1,\dots,u_i = z,u_{i+1},\dots,u_k = y$ such that $z,y \in \partial (x)$ and $z \not = y$. This shows that $d(x,u_{i+1})> d(x,z)$, contradicting the fact that $z$ is a boundary vertex of $x$. Hence $\partial (x)$ is an $x$-position set.
\end{proof}
It follows from Proposition~\ref{pro:boundary} that for any $x \in V(G)$ the set $\ext (G) \setminus \{ x\} $ is an $x$-position set. The bound in Proposition~\ref{pro:boundary} is tight for the vertex $x$ in the graph in Figure~\ref{fig:boundary}, but this is not true in general. In fact for any $s \geq t$ if we take a vertex $x$ in the partite set of order $t+1$ in the complete bipartite graph $K_{s,t+1}$ then $|\partial (x)| = t$ but $p_x(K_{s,t+1}) = s$.

\begin{figure}
			\centering
			\begin{tikzpicture}[x=0.4mm,y=-0.4mm,inner sep=0.2mm,scale=0.8,very thick,vertex/.style={circle,draw,minimum size=15,fill=white}]
			\node at (0,0) [vertex,color=red] (u1) {};
			\node at (50,0) [vertex] (u2) {};
			\node at (100,0) [vertex,color=red] (u3) {};
			\node at (150,0) [vertex] (u4) {};
			\node at (200,0) [vertex] (u5) {};
			\node at (250,0) [vertex,color=red] (u6) {};
			\node at (100,-50) [vertex] (y) {};
			\node at (100,-100) [vertex] (x) {$x$};
			
			\path
			
		(u1) edge (u2)
		(u2) edge (u3)
		(u3) edge (u4)
		(u4) edge (u5)
		(u5) edge (u6)
		(u2) edge (y)
		(u3) edge (y)
		(u4) edge (y)
		(x) edge (y)
				;
			\end{tikzpicture}
			\caption{The boundary $\partial (x)$ (shown in red) is a $p_x$-set}
			\label{fig:boundary}
		\end{figure}
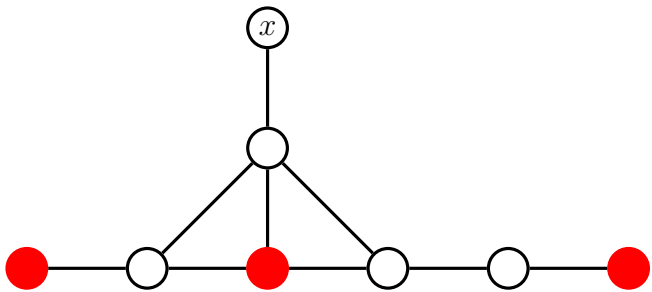

Finally we present a Nordhaus-Gaddum relation for the vertex position number. 
\begin{theorem}\label{thm:Nordhasu Gaddum}
For any graph $G$ we have $n-1\leq \vp (G) +\vp (\overline G) \leq 2n-1$. Both bounds are tight.
\end{theorem}
\begin{proof}
Notice that one of $G$ and $\overline {G}$ could be disconnected. Let $x$ be any vertex of a graph $G$ with degree $d(x)$. In the complement $\overline G$ the vertex $x$ has degree $d'(x) = n-1-d(x)$. By Lemma~\ref{lem:degree bound} we thus have
\begin{equation}\label{eqn:NordhausGaddum}
   \vp (G) + \vp(\overline G) \geq p_x(G)+p_x(\overline G) \geq d(x) + n-1-d(x) = n-1. 
\end{equation}
To show that equality holds, consider the cycle $C_n$ for $n \geq 5$. If $n = 5$ the result is simple, as $\overline {C_5} \cong C_5$, so take $n \geq 6$. Label the vertices of the cycle $x_0,x_1,\dots ,x_{n-1}$, where $x_0 \sim x_{n-1}$ and $x_i \sim x_{i+1}$ for $0 \leq i \leq n-2$. As will be shown in Corollary~\ref{cor:cycles}, we have $\vp(C_n) = 2$. Consider the vertex $x_0$ (as $\overline C_n$ is vertex-transitive the choice is arbitrary) and let $S$ be a largest $x_0$-position set. The degree of $x_0$ is $d(x) = n-3$, so that by Lemma~\ref{lem:degree bound} we have $\vp (\overline C_n) \geq n-3$. Suppose that a vertex $x_i$, $3 \leq i \leq n-3$, belongs to $S$; then as $x_0,x_i,x_1$ and $x_0,x_i,x_{n-1}$ are shortest paths we must have $x_1,x_{n-1} \not \in S$, so that $|S| \leq n-3$. Furthermore $S$ cannot contain both vertices $x_{n-2}$ and $x_1$ and likewise cannot contain both $x_2$ and $x_{n-1}$, so in any case $|S| \leq n-3$. Thus $\vp (\overline C_n) = n-3$ and we have $\vp (C_n) + \vp(\overline C_n) = n-1$.

Applying this argument to a vertex of maximum degree $\Delta $ and a vertex with minimum degree $\delta $ gives the stronger bound $\vp (G) + \vp (\overline {G}) \geq n-1+\Delta -\delta $, so we see that we have equality in Equation~\ref{eqn:NordhausGaddum} if and only if $G$ is regular and both $G$ and $\overline G$ have vertex position number equal to their maximum degree.

Trivially for $n \geq 2$ we have $\vp (G) \leq n$, with equality if and only if $G$ has an isolated vertex. Not both of $G$ and $\overline G$ can have an isolated vertex, for if $G$ has an isolated vertex, then $\overline G$ contains a universal vertex. Therefore we do not have equality in both $\vp (G) \leq n$ and $\vp (\overline G) \leq n$, so it follows that $\vp(G) +\vp(\overline G) \leq 2n-1$. Equality holds if and only if $G$ contains an isolated vertex or a universal vertex.
\end{proof}		

\section{Vertex position numbers of certain classes of graphs}\label{Section: graph classes}
In this section, we determine the $x$-position number of certain standard classes of graphs.
 
\begin{lemma}\label{lem:constant distance}
Let $x$ be a vertex of a connected graph $G$ and $S_x \subseteq V(G)$ be an $x$-position set of $G$. If $C_1,C_2,\dots,C_k$ are the components of $G[S_x]$, then there exist $r_1,r_2,\dots ,r_k$ such that $d(x,y) = r_i$ for all $y\in C_i$.
\end{lemma}
\begin{proof}
Suppose that there is a component $C$ of $G[S_x]$ and $y,y' \in V(C)$ such that $d(x,y) \not = d(x,y')$. Then, considering a shortest path from $y$ to $y'$ in $G[C]$, we see that there is a pair $z,z' \in V(C)$ such that $z \sim z'$ and $d(x,z') = d(x,z)+1$. However, this implies that an $x,z$-geodesic followed by the edge $z \sim z'$ is a shortest $x,z'$-path that passes through $z$, a contradiction.  
\end{proof}
\begin{theorem}\label{thm:bipartite graphs}
If $G$ is a bipartite graph, then $\vp(G)\leq \alpha(G)$.
\end{theorem}
\begin{proof}
Let $S_x$ be an $x$-position set of $G$ and suppose for a contradiction that $S_x$ is not an independent set. Then there are $y,z \in S_x$ such that $y \sim z$ in $G$. By Lemma~\ref{lem:constant distance} we have $d(x,y) = d(x,z) = r$ for some $r \geq 1$. A shortest $x,y$-path, a shortest $x,z$-path and the edge $y \sim z$ together constitute an odd circuit, implying the existence of an odd cycle; since $G$ is bipartite, this is impossible and it follows that $p_x(G) \leq \alpha (G)$.
\end{proof}

\begin{theorem}\label{thm:multipartite graphs}
For $r \geq 2$, let $K_{n_1,n_2,\dots ,n_r}$ be the complete multipartite graph with partite sets $V_1,V_2,\dots, V_r$, where $n_i = |V_i|$ and $n_1 \geq n_2 \geq \dots \geq n_r$. Set $n = \sum _{i=1}^{r}n_i$. Then if the vertex $x$ lies in $V_i$, the $x$-position number is given by
\[ p_x(K_{n_1,n_2,\dots ,n_r}) = \max \{ n-n_i,n_i-1\} .\]
Thus $\vp (K_{n_1,n_2,\dots ,n_r}) = n-n_r$.
\end{theorem}
\begin{proof}
Let $x \in V_i$ and let $S$ be a maximum $x$-position set of the graph. Set $V = \bigcup_{i = 1}^{r} V_i$. Suppose that $S$ contains a vertex $y \in V_i\setminus \{ x\} $; for any vertex $z \in V \setminus V_i$ the path $x,z,y$ is a geodesic, so that in this case $S \cap (V \setminus V_i) = \emptyset $. Thus either $S \subseteq V\setminus V_i$ or $S \subseteq V_i\setminus \{ x\} $. Conversely, both of these sets are $x$-position sets by the argument of Lemma~\ref{eccentricity bound}, which yields the claimed bounds.
\end{proof}
Theorem~\ref{thm:multipartite graphs} shows that equality holds in the bound of Theorem~\ref{thm:bipartite graphs} for all complete bipartite graphs.

\begin{lemma}
Let $G$ be a connected graph of order $n$. Then for each $x\in V(G)$ there is a maximum $x$-position set without cutvertices.
\end{lemma}
\begin{proof}
Suppose that there is a maximum $x$-position set $M$ containing a cutvertex $v \neq x$ of $G$. Let $C_1,C_2,\dots,C_k$ be the components of $G\setminus \{v\}$, where $k \geq 2$. Without loss of generality we may assume that $x\in V(C_1)$. Then it follows that $M\cap V(C_i)=\emptyset$ for all $i=2,3,\dots,k$. Let $u_i$ be any vertex in $C_i$ for all $i=2,3,\dots,k$. If $k\geq 3$, then the set $M^\prime = (M\setminus \{v\})\cup \{u_2,u_3,\dots,u_k\}$ is an $x$-position set with order greater than $M$, a contradiction to the maximality of $M$. Hence $k = 2$. Let $u$ be a farthest vertex from $v$ in $C_2$. Then $u$ is not a cutvertex in $G$. Moreover, $M^\prime = (M\setminus \{v\})\cup \{u\}$ is a maximum $x$-position set containing fewer cutvertices than $M$; this implies the existence of a maximum $x$-position set without cutvertices in $G$.
\end{proof}

\begin{theorem}
For any block graph G,
 \[ p_x(G)= \begin{cases}
 \s(G)-1, &\text{if $x$ is a simplicial vertex, } \\
 \s(G), & \text{otherwise.}
 \end{cases}\]
\end{theorem}
\begin{proof}
First suppose that $x$ is a simplicial vertex. Then it is clear that $\s(G)\setminus \{x\}$ is an $x$-position set in $G$. Hence $p_x(G)\geq \s(G)-1$. On the other hand, in a block graph each vertex is either a cutvertex or a simplicial vertex. By the above lemma, $G$ contains a maximum $x$-position set without cutvertices. Hence the result follows.
\end{proof}

\begin{corollary}\label{cor:trees}
For any tree $T$ with $\ell $ leaves we have
\[ p_x(T)=\begin{cases}
 \ell -1, &\text{if $x$ is a leaf, } \\
 \ell , & \text{otherwise.}
 \end{cases}\]
\end{corollary}
Corollary~\ref{cor:trees} implies the following bound for the vertex position numbers in terms of the girth of the graph.
\begin{theorem}\label{girth bound}
If a graph $G$ has girth $g$ and minimum degree $\delta \geq 2$ and there are $N$ vertices at distance less than or equal to $\left \lfloor \frac{g-1}{2} \right \rfloor -1$ from a vertex $u$, then $p_u(G) \leq n-N$.
\end{theorem}
\begin{proof}
Set $r = \left \lfloor \frac{g-1}{2} \right \rfloor$. Fix a vertex $u$ of $G$ and consider the subgraph $G'$ induced by the vertices at distance at most $r$ from $u$. $G'$ is isomorphic to a tree, possibly with some edges added between the vertices at distance $r$ from $u$. It follows from Corollary~\ref{cor:trees} that the largest number of vertices from $G'$ that can belong to a $u$-position set is the number of vertices at distance exactly $r$ from $u$; hence there are at least $N$ vertices missing from any $u$-position set and $p_u(G) \leq n-N$.
\end{proof}
\begin{figure}
		\centering
		\begin{tikzpicture}[x=0.2mm,y=-0.2mm,inner sep=0.2mm,scale=0.7,thick,vertex/.style={circle,draw,minimum size=10}]
			\node at (180,200) [vertex] (v1) {$x$};
			\node at (8.8,324.4) [vertex] (v2) {};
			\node at (74.2,525.6) [vertex,fill=red] (v3) {};
			\node at (285.8,525.6) [vertex,fill=red] (v4) {};
			\node at (351.2,324.4) [vertex] (v5) {};
			\node at (180,272) [vertex] (v6) {};
			\node at (116.5,467.4) [vertex,fill=red] (v7) {};
			\node at (282.7,346.6) [vertex,fill=red] (v8) {};
			\node at (77.3,346.6) [vertex,fill=red] (v9) {};
			\node at (243.5,467.4) [vertex,fill=red] (v10) {};

			\path
			(v1) edge (v2)
			(v1) edge (v5)
			(v2) edge (v3)
			(v3) edge (v4)
			(v4) edge (v5)
			
			(v6) edge (v7)
			(v6) edge (v10)
			(v7) edge (v8)
			(v8) edge (v9)
			(v9) edge (v10)
			
			(v1) edge (v6)
			(v2) edge (v9)
			(v3) edge (v7)
			(v4) edge (v10)
			(v5) edge (v8)

			;
		\end{tikzpicture}
		\caption{A largest $x$-position set (red vertices) in the Petersen graph}
		\label{fig:Petersen}
	\end{figure}
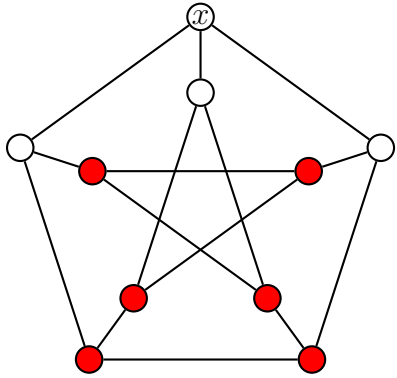 
	
Theorem~\ref{girth bound} is tight for the Petersen and Hoffman-Singleton graphs by Lemma~\ref{eccentricity bound}. Finally, we note that Lemma~\ref{lem:degree bound} gives the vertex position number of sufficiently large Kneser graphs. The Kneser graph $K(n,k)$ is the graph with vertex set equal to all $k$-subsets of $\{ 1,2,\dots ,n\} $ with an edge between any two such subsets if and only if they are disjoint.
\begin{theorem}\label{thm:Kneser}
For sufficiently large $n$ we have $\vp (K(n,k)) = \binom{n-k}{k}$.
\end{theorem}
\begin{proof}
For $n \geq 3k$ the Kneser graph $K(n,k)$ has diameter two. Note that $K(n,k)$ is vertex-transitive, so $\vp (K(n,k)) = \vp ^-(K(n,k))$ and we can without loss of generality consider the vertex $\{ 1,2,\dots ,k\} $; let $S$ be a largest $\{ 1,2,\dots ,k\} $-position set.  Lemma~\ref{lem:degree bound} gives $|S| = \vp (K(n,k)) \geq \binom{n-k}{k}$. Suppose that $S$ contains a vertex at distance two from $\{ 1,2,\dots ,k\} $, say $\{ 1,2,\dots ,i,j_{i+1},j_{i+2},\dots ,j_k\} $, where \[ \{ 1,2,\dots ,k\} \cap \{ 1,2,\dots ,i,j_{i+1},j_{i+2},\dots ,j_k\} = \{ 1,2,\dots ,i\} . \] As $\{ 1,2,\dots ,k\} $ and $\{ 1,2,\dots ,i,j_{i+1},j_{i+2},\dots ,j_k\} $ have exactly $\binom{n-2k+i}{k}$ common neighbours, we would have $|S| \leq \binom{n}{k}-\binom{n-2k+i}{k}$, which is a polynomial of degree $k-1$, whereas the vertex degree is a polynomial function of $n$ with degree $k$. Thus for sufficiently large $n$ compared to $k$ the bound in Lemma~\ref{lem:degree bound} is best possible.
\end{proof}
Interestingly Theorem~\ref{thm:Kneser} implies that for large $n$ the vertex position number of $K(n,k)$ is significantly larger than the general position number, as given in~\cite{GhoMaiMahMomKlaRus}. For small $n$ Lemma~\ref{lem:degree bound} is not optimal; for example, as previously noted for the Petersen graph $P$ (isomorphic to $K(5,2)$) the second neighbourhood of a vertex is a largest vertex position set and $\vp (P) = \gp (P) = 6$.

\section{Characterisation results}\label{Section: characterisation}

We now make use of the bounds derived in Section~\ref{Section: vertex position sets} to characterise graphs with very large or very small vertex position numbers.

\begin{corollary}\label{cor:position num n-1}
A connected graph $G$ with order $n$ satisfies $\vp(G) = n-1$ if and only if $G$ contains a universal vertex, whilst $\vp^-(G) = n-1$ if and only if $G$ is a complete graph.
\end{corollary}
\begin{proof}
By Lemma~\ref{lem:degree bound} any universal vertex $u$ has vertex position number $p_u(G) = n-1$. Conversely, by Lemma~\ref{lem:eccentricity upper bound} any vertex $u$ with $p_u(G) = n-1$ has eccentricity one and hence is universal. If $\vp ^-(G) = n-1$, it follows that $G$ is complete.  
\end{proof}

\begin{corollary}\label{cor:cycles}
A connected graph $G$ satisfies $\vp^-(G) = \vp(G) = 2$ if and only if $G$ is a cycle. Similarly $\vp(G) = 2$ only for cycles and paths of length $\geq 2$.
\end{corollary}
\begin{proof}
Let $C_n$ be a cycle for some $n \geq 3$. As $C_n$ is $2$-regular, by Lemma~\ref{lem:degree bound} we have $\vp^- (C_n) \geq 2$. We now show that $\vp (C_n) \leq 2$. Let $x$ be any vertex of $C_n$ and suppose that there exists an $x$-position set $S_x$ of $C_n$ of order $\geq 3$; we can assume that $x \not \in S_x$. Let $A$ be the set of antipodal vertices of $x$, i.e. the vertices of $C_n$ at distance $\lfloor \frac{n}{2} \rfloor $ from $x$. We have $|A| = 1$ if $n$ is even and $|A| = 2$ if $n$ is odd. Let $P_1$ and $P_2$ be the two shortest paths from $x$ to $A$.  Then one of $P_1$ and $P_2$ (say $P_2$) contains distinct vertices $u,v \in S_x$. Hence either $u$ is on the $v,x$-geodesic or $v$ is on the $u,x$-geodesic, a contradiction. Hence $|S_x| \leq 2$ and, since $x$ is an arbitrary vertex of $C_n$, we have $\vp ^-(C_n) = \vp (C_n) = 2$. 

Conversely, suppose that $\vp(G) = 2$; it follows from Lemma~\ref{lem:degree bound} that $G$ has maximum degree $\Delta = 2$, so that $G$ is either a path or a cycle. As $\vp^-(P_{\ell }) = 1$ for paths by Corollary~\ref{cor:path}, it follows that if $\vp^-(G) = 2$, then $G$ is a cycle. 
\end{proof}
Now we characterise some graphs with very large vertex position number.

\begin{lemma}\label{lem:vertex position n-2}
A vertex $u$ of a connected graph $G$ with order $n$ has $p_u(G) = n-2$ if and only if either i) $u$ has degree $d(u) = n-2$, or ii) $u$ has a neighbour $v$ such that $v$ is a cutvertex of $G$ and $\{ u,v\} $ dominates $G$, in which case the unique largest $p_u$-set of $G$ is $S_u = V(G) \setminus \{ u,v\} $.
\end{lemma}
\begin{proof}
Let $G$ be a graph with order $n$ and let $u \in V(G)$ satisfy $p_u(G) = n-2$, with $S_u$ a largest $u$-position set. If $u$ is a universal vertex, then by Corollary~\ref{cor:position num n-1} we would have $p_u(G) = n-1$, so $u$ has degree $d(u) \leq n-2$. Hence by Lemma~\ref{lem:eccentricity upper bound} $u$ has eccentricity two. If $d(u) = n-2$, then $N(u)$ is a $u$-position set, so we can assume that $d(u) \leq n-3$. Hence $V(G) = \{ u\} \cup N(u) \cup N^2(u)$ and $|N^2(u)| \geq 2$.  

Let $x,y$ be any vertices in $N^2(u)$. Suppose that one of these vertices, say $x$, has at least two common neighbours with $u$. If $x \in S_u$, then we must have $N(x) \cap N(u) \cap S_u = \emptyset $; however, this contradicts $u \not \in S_u$, implying that $x \not \in S_u$ and $S_u = V(G) \setminus \{ u,x\}$, so that $N(u) \subseteq S_u$. As $y \in N^2(u) \cap S_u$ and has a neighbour in $N(u)$, this is also a contradiction. Therefore every vertex in $N^2(u)$ has just one neighbour in $N(u)$. 

Suppose that $N(u) \cap N(x) = \{ v\} $ and $N(u) \cap N(y) = \{ v'\} $, where $v \not = v'$. Then $|S_u \cap \{ v,x\} | \leq 1$ and $|S_u \cap \{ v',y\} | \leq 1$, which, together with $u$, accounts for at least three vertices missing from $S_u$. Thus we must have $v = v'$ and there is a vertex $v \in N(u)$ that is the unique neighbour in $N(u)$ of every vertex in $N^2(u)$. Hence $\{ u,v\} $ dominates $G$. Furthermore in any such graph $V(G) \setminus \{ u,v\} $ is a $u$-position set, so that $\vp(G) = n-2$. 
\end{proof}

\begin{theorem}\label{thm:K_2,2,2,2}
For $n \geq 4$, a graph $G$ with order $n$ satisfies $\vp ^-(G) = \vp (G) = n-2$ if and only if $G$ is isomorphic to an even clique with a perfect matching deleted, i.e. if and only if $G \cong K_{2,2,\dots ,2}$.
\end{theorem}
\begin{proof}
Assume that $G$ is a graph such that $p_u(G) = n-2$ for every $u \in V(G)$. Suppose that $G$ contains a vertex $u$ with degree $\leq n-3$, so that by Lemma~\ref{lem:vertex position n-2} $u$ has eccentricity two and $u$ has a neighbour $v$ that is a cutvertex of $G$ and $\{ u,v\} $ is a dominating set of $G$. As $\vp (G) = n-2$, $G$ contains no universal vertex, so that there is a neighbour $w$ of $u$ such that $v \not \sim w$. Hence if $x \in N^2(u)$ we have $d(x,w) \geq 3$ and $x$ has eccentricity at least three, so that by Lemma~\ref{lem:eccentricity upper bound} we have $p_x(G) \leq n-3$, a contradiction. Therefore every vertex of $G$ has degree $n-2$ and $G$ is isomorphic to a $K_{2r}$ with a perfect matching deleted. Conversely in such a graph every vertex has vertex position number $n-2$. 
\end{proof}

\begin{theorem}
For $n \geq 4$, a graph $G$ has $\vp ^-(G) = n-2$ and $\vp (G) = n-1$ if and only if i) $G$ is isomorphic to a clique with a non-empty, non-perfect matching deleted or ii) $G$ is the join of $K_1$ with a disjoint union of cliques.
\end{theorem}
\begin{proof}
Let $G$ be a graph with $\vp ^-(G) = n-2$ and $\vp (G) = n-1$. We can assume that $G$ contains $r \geq 1$ universal vertices as well as at least two vertices with degree $\leq n-2$. If every vertex has degree either $n-1$ or $n-2$, then $G$ is isomorphic to a clique with a matching deleted. To avoid the graph having $\vp ^-(G) = \vp(G) = n-1$ the matching is non-empty and to avoid having $\vp ^-(G) = \vp(G) = n-2$ the matching is not perfect by Theorem~\ref{thm:K_2,2,2,2}.

Suppose that $G$ contains a vertex $u$ with $p_u(G) = n-2$ and degree $d(u) \leq n-3$. By Lemma~\ref{lem:vertex position n-2} $u$ has eccentricity two and has a neighbour $v$ that is a cutvertex. By Lemma~\ref{lem:eccentricity upper bound} every vertex of $G$ has eccentricity at most two, so, considering the vertices in $N^2(u)$, we see that $v$ is a universal vertex. As $v$ is a cutvertex, it is the unique universal vertex of $G$, so that every other vertex $w$ of $G$ must have $p_w(G) = n-2$. For any $w \in V(G)\setminus \{ v\} $, let $S_w$ be a $w$-position set of order $n-2$. Suppose that $v \in S_w$. Then $S_w$ cannot contain $w$ or any vertex from a component of $G -v$ not containing $w$; this impossible unless $G \cong P_3$, which has the stated structure, so we can assume that $S_w = V(G) \setminus \{ w,v\} $ for each $w \in V(G) \setminus \{ v\} $. Suppose then that there is a component $W$ of $G-v$ that is not a clique, so that there are vertices $w_1,w_2,w_3$ in $W$ such that $w_1,w_2,w_3$ is a path, but $w_1 \not \sim w_3$. However, this implies that $w_2$ and $w_3$ cannot both belong to a largest $w$-position set $S_{w_1}$, which is a contradiction, since $S_{w_1} = V(G) \setminus \{ w_1,v\} $. Hence the graph must be the join of $K_1$ with a disjoint union of cliques, which is easily verified to have the correct vertex position numbers.
\end{proof}

\section{Computational complexity}\label{Section: complexity}
Given a graph $G$, in this section we show that the vertex position number $p_x(G)$ can be computed in polynomial time for any vertex $x\in V(G)$. In particular, we will show that $p_x(G)$, for each $x\in V(G)$, can be computed as an independent set calculated on a graph obtained as a transformation of $G$.
To this aim we need the following definitions.
\begin{definition}
A graph is a \emph{comparability graph} if the edges connect pairs of elements that are comparable to each other in a partial order.
\end{definition}

\begin{definition}
Given a graph $G$ and a vertex $x\in V(G)$, the \emph{reduced} graph $\widetilde{G}_x$ is the graph on the same vertices $V(G)$ obtained from $G$ by removing all the edges connecting vertices at the same distance from $x$.
\end{definition}

\begin{definition}
Given a graph $G$ and a vertex $x\in V(G)$, the graph $G^*_x$ is the graph on the same vertices $V(G)$ obtained from the reduced graph $\widetilde{G}_x$ by adding an edge between any two vertices of any geodesic to $x$.
\end{definition}

See Figure~\ref{fig:reduced} for a visualisation of $\widetilde{G}_x$ and $G^*_x$, starting from a graph $G$ and a vertex $x$.
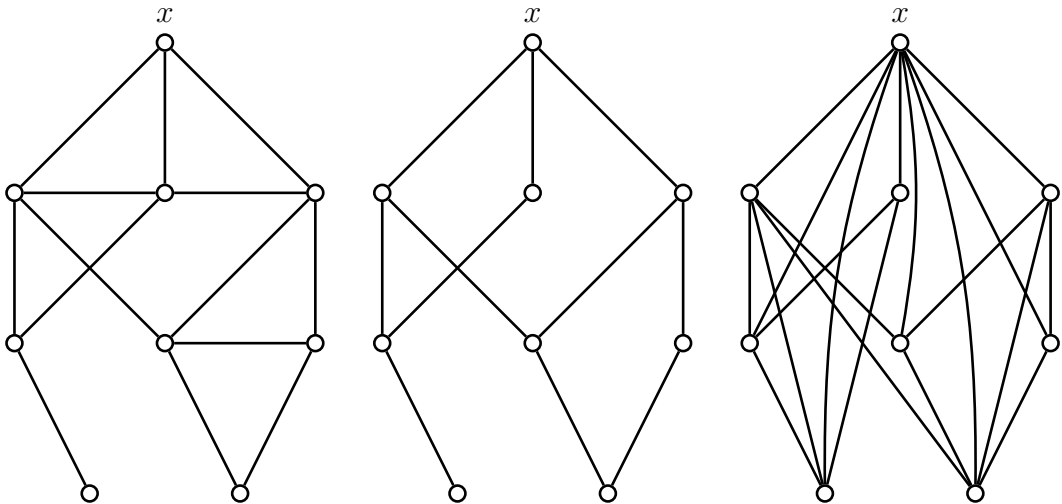
\begin{figure}
\centering
\begin{tikzpicture}
\pgfsetlinewidth{1pt}	
	\vertex (v_1) at (2,6) [label=above:$x$] {}; 
	\vertex (v_2) at (4,4) {};  
	\vertex (v_3) at (4,2) {};
	\vertex (v_4) at (0,2) {};
	\vertex (v_5) at (0,4) {};
	\vertex (v_6) at (2,4) {};
	\vertex (v_7) at (2,2) {};
	\vertex (v_8) at (1,0) {};
	\vertex (v_9) at (3,0) {};
	\path
		(v_1) edge (v_2)
		(v_2) edge (v_6)
		(v_6) edge (v_1)
		(v_1) edge (v_5)
		(v_5) edge (v_6)
		(v_6) edge (v_4)
		(v_4) edge (v_5)
		(v_5) edge (v_7)		
		(v_7) edge (v_2)		
		(v_2) edge (v_3)		
		(v_3) edge (v_7)		
		(v_7) edge (v_9)		
		(v_9) edge (v_3)		
		(v_4) edge (v_8)		
			 ; 
\end{tikzpicture}~~~~~~
\begin{tikzpicture}
\pgfsetlinewidth{1pt}	
	\vertex (v_1) at (2,6) [label=above:$x$] {}; 
	\vertex (v_2) at (4,4) {};  
	\vertex (v_3) at (4,2) {};
	\vertex (v_4) at (0,2) {};
	\vertex (v_5) at (0,4) {};
	\vertex (v_6) at (2,4) {};
	\vertex (v_7) at (2,2) {};
	\vertex (v_8) at (1,0) {};
	\vertex (v_9) at (3,0) {};
	\path
		(v_1) edge (v_2)
		(v_6) edge (v_1)
		(v_1) edge (v_5)
		(v_6) edge (v_4)
		(v_4) edge (v_5)
		(v_5) edge (v_7)		
		(v_7) edge (v_2)		
		(v_2) edge (v_3)		
		
		(v_7) edge (v_9)		
		(v_9) edge (v_3)		
		(v_4) edge (v_8)		
			 ;

\end{tikzpicture}~~~~~~
\begin{tikzpicture}
\pgfsetlinewidth{1pt}	
	\vertex (v_1) at (2,6) [label=above:$x$] {}; 
	\vertex (v_2) at (4,4) {};  
	\vertex (v_3) at (4,2) {};
	\vertex (v_4) at (0,2) {};
	\vertex (v_5) at (0,4) {};
	\vertex (v_6) at (2,4) {};
	\vertex (v_7) at (2,2) {};
	\vertex (v_8) at (1,0) {};
	\vertex (v_9) at (3,0) {};
	\path
		(v_1) edge (v_2)
		(v_6) edge (v_1)
		(v_1) edge (v_5)
		(v_6) edge (v_4)
		(v_4) edge (v_5)
		(v_5) edge (v_7)		
		(v_7) edge (v_2)		
		(v_2) edge (v_3)		
		
		(v_7) edge (v_9)		
		(v_9) edge (v_3)		
		(v_4) edge (v_8)		
			 ; 

	\path
		(v_8) edge (v_5)
		(v_8) edge (v_6)
		(v_9) edge (v_5)
		(v_9) edge (v_2)
		(v_1) edge (v_3)
		(v_1) edge (v_4)
		(v_1) edge [bend left=10] (v_7)
		(v_1) edge [bend right=10] (v_8)
		(v_1) edge [bend left=10] (v_9)
			 ; 
    
\end{tikzpicture}
\caption{~\label{fig:reduced}From left to right: a graph $G$ with a vertex $x$, the graph $\widetilde{G}_x$, and the graph $G^*_x$. 
The graphs $G$ and $\widetilde{G}_x$ are drawn by placing the vertices at the same distance from $x$ on a common horizontal level.}
\end{figure}

\begin{lemma}
Given a graph $G$ and a vertex $x\in V(G)$, $G^*_x$ is a comparability graph.
\end{lemma}
\begin{proof}
The partial order underlying the graph  $G^*_x$ consists of its vertices and, by definition, two vertices $u,v\in V(G^*_x)$  are such that $u < v$ if $u$ and $v$ are on the same geodesic to $x$ and $d(u,x)<d(v,x)$. Since $G^*_x$ has been built from $\widetilde{G}_x$ by adding an edge between every pair of vertices on each geodesic to $x$, then any two comparable vertices are adjacent and hence $G^*_x$ is a comparability graph.
\end{proof}

\begin{lemma}\label{lem:equiv}
Given a graph $G$ and a vertex $x\in V(G)$, $S_x$ is an $x$-position set of $G$ if and only if $S_x$ is an $x$-position set for $\widetilde{G}_x$. Then $p_x(G)=p_x(\widetilde{G}_x)$.
\end{lemma}
\begin{proof}
Note that any geodesic to $x$ in $G$ is also a geodesic to $x$ in $\widetilde{G}_x$. Then  $S_x$ is a $x$-position set for $G$ if and only if there are no two vertices on the same geodesic to $x$ in $G$, that is, if and only if there are no two vertices on the same geodesic to $x$ in $\widetilde{G}_x$. Hence,  if and only if $S_x$ is a $x$-position set for $\widetilde{G}_x$. As consequence, any maximum vertex position set of $G_x$ is a maximum vertex position set for $\widetilde{G}_x$. Then $p_x(G)=p_x(\widetilde{G}_x)$.
\end{proof}

Given a graph $G$, let us denote the graph induced by vertices in $V(G)\setminus \{x\}$ as $G - x$.
\begin{lemma}\label{lem:is}
Given a graph $G$ and a vertex $x\in V(G)$, $S_x$ is an $x$-position set of $G$ if and only if $S_x$ is an independent set of $G^*_x - x$. Then $p_x(G)=\alpha(G^*_x - x)$.
\end{lemma}
\begin{proof}
Let $S_x$ an $x$-position set in $G$. By Lemma~\ref{lem:equiv}, $S_x$ is also a $x$-position set in $\widetilde{G}_x$. Assume that $S_x$ is not an independent set in $G^*_x - x$. Then there are two adjacent vertices $u,v\in V(G^*_x - x)\cap S_x$. By definition of $G^*_x$, $u$ and $v$ are on the same geodesic to $x$ in $\widetilde{G}_x$ and then in $G$, a contradiction. 


Assume now that $S_x$ is an independent set of $G_x^* - x$. Let $u\in S_x$ and let $P$ be any $u-x$ geodesic in G. Then it follows from the construction of $G_x^*$ that $u$ is adjacent to all the vertices of $V(P-u)$ in $G_x^*$. This immediately shows that $V(P)\cap S_x=\{u\}$. Consequently,  $S_x$ is a $x$-position set in $G$ and hence $p_x(G)=\alpha(G_x^*-x)$.
\end{proof}

\begin{algorithm}[ht]
\SetKwInput{Proc}{Algorithm A}
\Proc{}
\SetKwInOut{Input}{Input}
\Input{A connected graph $G$, a vertex $x\in V(G)$}
\SetKwInOut{Output}{Output}
\Output{A maximum $x$-position set $S_x$, and $p_x(G)$ 
}
\BlankLine
\BlankLine
Let $D[u] :=d(u,x)$, for each $u\in V(G)$\;\label{line:dist}  
        
\For {each $uv \in E(G)$}{\label{line:rem} 
        \If{$D[u]=D[v]$}{
                   remove $uv$ from $E(G)$\;}   
        }

\For{each $u \in V(G-x)$}{\label{line:add}
Let $Q$ be a queue and $R:=\{u\}$\;\label{line:QR}
$Q.enqueue(u)$\;
\While{$Q$ is not empty  \label{line:ciclo}}
   { $v := Q.dequeue()$\; \label{line:deq}
     \For {each $w$ in $N_G(v)\setminus R$ such that $D[w]>D[v]$}{
        $Q.enqueue(w)$\;
        $R:= R\cup \{w\}$\;
        \If{$D[w]>D[u]+1$}{
               add $uv$ to $E(G)$\; \label{line:uv}
               }    
        }
     }                  
        
    }

Let $S$ a maximum independent set of $G-x$\;  \label{line:is}              
                       
\Return $S$, $|S|$
 
\caption{Algorithm A to compute a maximum $x$-position set $S_x$ of a graph $G$ and $p_x(G)$, for a given $x\in V(G)$.
}
\label{alg:A}
\end{algorithm}

\begin{theorem}\label{theo:comp}
Given a graph $G$ and a vertex $x\in V(G)$, a maximum $x$-position set can be computed in $O(n m \log(n^2/m))$ time, where $n=|V(G^*_x)|$ and $m=|E(G^*_x)|$.
\end{theorem}
\begin{proof}
Algorithm A in Figure~\ref{alg:A} compute the distances of each vertex $v\in V(G)$ from $x$ at Line~\ref{line:dist}. This requires $O(n+m)$ time. 
With the loop at Line~\ref{line:rem}, $\widetilde{G}_x$ is computed from $G$ by removing edges between vertices at the same distance from $x$. This requires $O(m)$ time. The loop at Line~\ref{line:add} add edges to the graph in order to build $G^*_x$. This requires $O(n^2+nm)$ time since Lines from~\ref{line:QR} to~\ref{line:uv} codifies for a 
breadth-first visit of the vertices on a geodesic to $x$ passing through a vertex $u$. This visit, based on a queue $Q$ and a set $R$ of the visited vertices, requires $O(n+m)$ time and since it is repeated for each vertex $u$ in $V(G-x)$, the total time is $O(n^2+nm)$.
Finally, at Line~\ref{line:is} an independent set $S$ of the resulting comparability graph $G^*-x$ is computed. According to~\cite{golumbic}, the computation of an independent set for a comparability graph requires $O(n m \log(n^2/m))$. The last step determines the computational time of the whole algorithm. By Lemma~\ref{lem:is}, the set $S$ is also a $x$-position set of $G$, then Algorithm A correctly returns $S$ and its order.
\end{proof}

\begin{corollary}
Given a graph $G$, $\vp^-(G)$ and $\vp(G)$ can be computed in $O(n^4 \log(n))$ time, where $n=V(G)$.

\end{corollary}
\begin{proof}
Given a graph $G$, by calling Algorithm A for each vertex $x$ of $G$, $\vp^-(G)$ and $\vp(G)$ can be easily computed. Since by Theorem~\ref{theo:comp} each call requires  $O(n m \log(n^2/m))$, where $n=|V(G^*_x)|$ and $m=|E(G^*_x)|$. Considering that $n=|V(G^*_x)|=|V(G)|$ and $m=O(n^2)$, each call requires $O(n^3\log(n))$ time, for a total of $O(n^4\log(n))$ time.
\end{proof}

\section*{Acknowledgements}
The fourth author gratefully acknowledges funding support from EPSRC grant EP/W522338/1 and London Mathematical Society grant ECF-2021-27.

\end{document}